\documentclass[11pt]{amsart}

\usepackage{graphicx}
\graphicspath{ {images/} }
\usepackage[all,cmtip]{xy}
\usepackage{amssymb,verbatim,amscd,amsfonts,amsmath,amsthm,enumerate}
\usepackage[all]{xy}
\usepackage{subfig}
\usepackage{tikz}
\usetikzlibrary{shapes,arrows,shadows}
\usetikzlibrary{decorations.markings}
\usepackage{color}
\usepackage[normalem]{ulem}

\usepackage{hyperref}
\usepackage{wrapfig}
\usetikzlibrary{arrows}
\usepackage{pinlabel}

\long\def\symbolfootnote[#1]#2{\begingroup%
\def\thefootnote{\fnsymbol{footnote}}\footnote[#1]{#2}\endgroup}

\newtheorem{theorem}{Theorem}

\newtheorem{proposition}[theorem]{Proposition}
\newtheorem{corollary}[theorem]{Corollary}
\newtheorem{lemma}[theorem]{Lemma}

\newtheorem{conjecture}{Conjecture}

\newtheorem{remark}{Remark}

\newtheorem*{remark*}{Remark}

\theoremstyle{definition}

\newtheorem*{namedtheorem}{\theoremname}
\newcommand{\theoremname}{testing}

\newcommand{\R}{\mathbb{R}}
\newcommand{\C}{\mathbb{C}}
\newcommand{\N}{\mathbb{N}}
\newcommand{\Z}{\mathbb{Z}}
\newcommand{\Q}{\mathbb{Q}}

\def\Aut{\operatorname{Aut}}
\def\SO{\operatorname{SO}}

\def\Out{\operatorname{Out}}
\def\Mod{\operatorname{Mod}}

\def\proj{\operatorname{proj}}

\newenvironment{dedication}
  {\thispagestyle{empty}
   \vspace*{\stretch{1}}
   \itshape             
   \raggedleft          
  }
  {\par 
   \vspace{\stretch{3}} 
  }

\begin{document}

\title[]{Quotients of the mapping class group by power subgroups}
\author{Javier Aramayona}
\address{Universidad Aut\'onoma de Madrid \& ICMAT \\ C. U. de Cantoblanco. 28049, Madrid, Spain}
 
 \author{Louis Funar} 
 \address{ Institut Fourier, UMR 5582, Laboratoire de Math\'ematiques \\
Universit\'e Grenoble Alpes, CS 40700, 38058 Grenoble cedex 9, France}

\date{\today}
\thanks{The first author was partially funded by grants RYC-2013-13008 and MTM2015-67781. }

%

\begin{abstract}
We study the quotient of the mapping class group $\Mod_g^n$ of a surface of genus $g$ with $n$ punctures, by the subgroup $\Mod_g^n[p]$ generated by the $p$-th powers of Dehn twists. 

Our first main result is that $\Mod_g^1 /\Mod_g^1[p]$ contains an infinite normal subgroup of infinite index, and in particular is not commensurable to a higher-rank lattice, for all but finitely many explicit values of $p$. 
Next, we prove that $\Mod_g^0/ \Mod_g^0[p]$ contains a K\"ahler subgroup of finite index, for every $p\ge 2$ coprime with six.
Finally, we observe that the existence of finite-index subgroups of $\Mod_g^0$ with infinite abelianization is equivalent to the analogous problem for  $\Mod_g^0/ \Mod_g^0[p]$.
\end{abstract}

\maketitle

\begin{dedication}
Dedicated to the memory of \c{S}tefan Papadima (1953-2018)
\end{dedication}

\section{Introduction and statements}

Throughout, $S_{g}^{n}$ denotes the connected orientable surface of genus $g\ge 2$, with empty boundary and $n\ge 0$ marked points. Associated to $S_g^n$ is the mapping class group $\Mod_g^n$, namely the group of  homeomorphisms of $S_g^n$ up to isotopy. In the case when $n=0$, we will drop it from the above notation and write $S_g$ and $\Mod_g$. 

There is widespread interest in studying homomorphisms from mapping class groups to compact Lie groups, notably through the so-called {\em quantum representations} of mapping class groups. As it turns out, if  $G$ is a compact Lie group and $\Mod_g^n \to G$ is a homomorphism (with $g\ge 3$) then the image of every Dehn twist has finite order 
\cite[Corollary 2.6]{AS-rig}. With this motivation in mind, it is natural to study the structure of the group \[\frac{\Mod_g^n}{\Mod_g^n[p]},\] where $\Mod_g^n[p]$ denotes the (normal) subgroup of $\Mod_g^n$ generated by $p$-powers of Dehn twists. We remark that  $\Mod_g^n/\Mod_g^n[p]$ is known to be infinite if $g=2$ and $p\ge 4$ \cite{Humphries}, or if $g\ge 3$ and $p\notin \{1,2,3,4,6,8,12\}$ \cite{F}. 

\subsection{Non-lattice properties.}

In \cite{FP},  Pitsch and the second author  used quantum topological techniques to prove that $\Mod_g/\Mod_g[p]$ is not commensurable to a higher-rank lattice whenever $g\ge 4$ and $p\ge 2g-1$. The first purpose of this note is to give the following uniform version of this result, in the case of once-punctured surfaces: 

\begin{theorem}
Let $g,p \ge 2$, where $p \ge 4$ if $g=2$, and $p\notin \{2,3,4,6,8,12\}$ if $g\ge 3$.
Then \[\frac{\Mod_g^1}{\Mod_g^1[p]}\] has a descending normal series $Q_1 \trianglerighteq Q_2\trianglerighteq \ldots$  such that $Q_{i+1}$ has infinite index in $Q_i$ for all $i\ge 1$.
\label{thm:main}
\end{theorem}

The groups $Q_i$ in the statement of Theorem \ref{thm:main} are obtained through a covering construction, similar in spirit to that of the ``non-geometric'' injective homomorphism between mapping class groups of \cite{ALS}.

In light of Margulis' Normal Subgroup Theorem \cite{Margulis}, Theorem \ref{thm:main} has the following immediate consequence: 

\begin{corollary}
Let $g,p \ge 2$, where $p \ge 4$ if $g=2$, and $p\notin \{2,3,4,6,8,12\}$ if $g\ge 3$. Then $\Mod_g^1/\Mod_g^1[p]$ is not commensurable to a lattice in a semisimple Lie group of real rank $\ge 2$.
\label{cor:main}
\end{corollary}

Next, we will give the following analogue of the {\em Birman short exact sequence} for the groups $\Mod_g^1/\Mod_g^1[p]$. Below, $\pi_1(S_g)[p]$ denotes the subgroup of $\pi_1(S_g)$ generated by $p$-powers of {\em simple} loops, that is, loops without transverse self-intersection.

\begin{proposition}\label{thm:surface}
 Let $g,p \ge 2$. There is an exact sequence: 
\[ 1\to A\to \frac{\pi_1(S_g)}{\pi_1(S_g)[p]}\to \frac{\Mod_g^1}{\Mod_{g}^1[p]}\to \frac{\Mod_{g}}{\Mod_{g}[p]}\to 1\]
where $A$ is central, namely $A\subseteq Z\left(\frac{\pi_1(S_g)}{\pi_1(S_g)[p]}\right)$.
\end{proposition}

Combining Proposition \ref{thm:surface} with the construction of the groups $Q_i$ of Theorem \ref{thm:main}, we will immediately obtain:

\begin{corollary}\label{cor:surfintermediary}
Let $g,p \ge 2$, where $p \ge 4$ if $g=2$ and $p\notin \{2,3,4,6,8,12\}$ if $g\ge 3$.  Then \[\frac{\pi_1(S_g)}{\pi_1(S_g)[p]}\] is not commensurable to a lattice in a semisimple Lie group of real rank $\ge 2$.
\label{cor:nonlatticepi}
\end{corollary}


\subsection{Virtually K\"ahler quotients} Farb \cite{Farb} has asked whether $\Mod_g$ (with $g\ge 3$) is a {\em K\"ahler group}, that is, the fundamental group of a compact K\"ahler manifold. Our next result shows that a large class of quotient groups ${\Mod_{g}}/{\Mod_{g}[p]}$ have finite-index subgroups which are K\"ahler: 

\begin{theorem}\label{thm:kahler} Suppose $p\ge 2$ and  $\gcd(p,6)=1$. For every $g\ge 0$, the group ${\Mod_{g}}/{\Mod_{g}[p]}$ is virtually K\"ahler, that is, it has a K\"ahler subgroup of finite index.
\end{theorem}

In order to prove Theorem \ref{thm:kahler}, we adapt arguments of Pikaart-Jong \cite{PJ} to exhibit a compact K\"ahler manifold whose fundamental group is a finite index subgroup of ${\Mod_{g}}/{\Mod_{g}[p]}$. 

As an immediate consequence, we obtain the following result; compare with Theorem \ref{thm:main}: 

\begin{corollary}\label{cor:rank1}
Suppose $p\ge 2$ and $\gcd(p,6)=1$. For every $g\ge 0$, the group $\Mod_g/\Mod_g[p]$  is not a lattice in $\SO(n,1)$, or more generally in a group of 
Hodge type.
\end{corollary}

Another consequence is the following:

\begin{corollary}\label{cor:homom}
Let $g\ge 0$. Given $p\ge 2$ with $\gcd(p,6)=1$, a finite index K\"ahler subgroup $J\subset \Mod_g/\Mod_g[p]$, a 
lattice $\Lambda\subset \SO(n,1)$ and any homomorphism 
\[ \Phi: \Mod_g/\Mod_g[p]\to \Lambda \subset \SO(n,1), n >2\]
then one of the following holds: 
\begin{enumerate}
\item $\Phi|_J$ factors through $\Z$; 
\item  $\Phi|_J$ factors through $\pi_1(S_h)$, for some $h\geq 2$;
\item  $\Phi|_J$ is trivial. 
\end{enumerate}
In the first two cases $\Phi|_J$ surjects onto $\Z$ and $\pi_1(S_h)$, respectively. 
\end{corollary} 

\begin{remark}
The proof of Theorem \ref{thm:kahler} will provide explicit  finite index K\"ahler subgroups $J$. Thus, given any finite index subgroup 
$G\subset  \Mod_g/\Mod_g[p]$, Corollary \ref{cor:homom} applies to the subgroup $J\cap G$. 
As will become evident from our arguments, the groups $J$  fall frequently in  case (3) above. 
\end{remark}

One provides linear representations of the fundamental group of a K\"ahler manifold, like those obtained in 
Theorem \ref{thm:kahler}, by considering variations of Hodge structures, in particular those obtained 
from families of complex algebraic varieties. We might wonder whether the 
restrictions of the quantum representations, known to factorise through $\Mod_g/\Mod_g[p]$,
arise geometrically as above.

\begin{conjecture}
Complexified quantum representations of $\Mod_g$  are locally rigid. 
\end{conjecture}

If true, this conjecture would imply that the quantum representation at a prime level  
is a complex direct factor of a  $\Q$-variation of Hodge structures, according to (\cite{Simpson}, Thm.5, p.56).

\subsection{Abelianization of finite-index subgroups}
A celebrated result of Powell \cite{Powell} asserts that $\Mod_g$ has trivial abelianization whenever $g\ge 3$. In stark contrast, the situation for finite index subgroups remains mysterious: a well-known question of Ivanov \cite{Ivanov-15} asks whether every finite-index subgroup of $\Mod_g$ has finite abelianization. 
Our final result states that this question is equivalent to the analogous problem for the quotient subgroups $\Mod_g/\Mod_g[p]$. More concretely, we have:  

\begin{theorem}\label{prop:abelianization}
Let $K\subset \Mod_g$  be a normal subgroup of finite index $n$.
Then $\dim H_1(K;\Q) >0$ if and only if the image $K(n)$ of $K$ into $\Mod_g/\Mod_g[n]$ is of 
finite index $n$  and $\dim H_1(K(n);\Q) >0$. 
\end{theorem}

\medskip

\noindent{\bf Acknowledgements.} The authors are grateful to M. Boggi, R. Coulon, F. Dahmani, 
P. Eyssidieux, P. Ha\"issinsky, C. Leininger, 
M. Mj, W. Pitsch, and J. Souto for conversations and to H. Wilton  for pointing out the need to include the group $A$ in the statement of Proposition \ref{thm:surface}. 

\section{Preliminaries} 
In this section we introduce the basic definitions and notation needed for the rest of this note. We refer the reader to the standard text \cite{FM} for a comprehensive introduction to these topics. 

As mentioned in the introduction, we denote by $S_{g}^n$ a connected orientable surface of genus $g\ge 0$ with empty boundary and with $n\ge 0$ marked points. Let $\Mod_g^n$ be the mapping class group of $S_g^n$, namely the group of self-homeomorphisms of $S_g^n$ up to homotopy. From now on, if $n=0$ we will drop it from the notation and  write $S_g$ and $\Mod_g$. 

\subsection{Power subgroups} We write $T_c$ for the right Dehn twist about the (isotopy class of a) simple closed curve $c \subset S_g^n$. Given a number $p\in \N$, consider the subgroup $\Mod_g^n[p]$ generated by all $p$-powers of Dehn twists. Observe that $\Mod_g^n[p]$ is a normal subgroup, as \[fT_c f^{-1} = T_{f(c)}\] for every simple closed curve $c\subset S_g^n$ and every $f\in \Mod_g^n$. 

The following result is due to Humphries \cite{Humphries} in the case $g=2$, and to the second author \cite{F} in the case $g\ge 3$: 

\begin{theorem}
Let $g\ge 2$ and $p\ge 0$. If $g= 2$, assume $p\ge 4$; if $g\ge 3$, assume that  $p\notin \{2,3,4,6,8,12\}$. Then $\Mod_g^n[p]$ has infinite index in $\Mod_g^n$. 
\label{thm:infindex}
\end{theorem}
Abusing notation, we will write \[\proj^p_g: \Mod_g^n \to \frac{\Mod_g^n}{\Mod_g^n[p]}\] for the natural projection. By  Theorem \ref{thm:infindex}, the image of $\proj^p_g$ is infinite whenever $g=2$ and $p\ge 4$, or if $g\ge 3$ and $p\notin \{2,3,4,6,8,12\}$.

\subsection{The Birman short exact sequence}
\label{sec:ses} There is an obvious surjective homomorphism \[\Mod_g^1 \to \Mod_g\] given by {\em forgetting} the marked point of $S_g^1$, call it $x_0$. This homomorphism is well-known to fit in the so-called {\em Birman short exact sequence}: 

\begin{equation}
1 \to \pi_1(S_g,x_0) \to \Mod_g^1 \to \Mod_g \to 1
\label{eq:birman}
\end{equation}
The injective homomorphism $p: \pi_1(S_g,x_0) \to\Mod_g^1$ of \eqref{eq:birman} is called the {\em point-pushing homomorphism}. Abusing notation, we will simply write $\pi_1(S_g)$ for $\pi_1(S_g,x_0)$; in addition, we will often not distinguish between $\pi_1(S_g)$ and its image in $\Mod_g^1$ under the point-pushing homomorphism. 

We will also need a description on the image of simple loops under the point-pushing homomorphism; we refer the reader to \cite[Section 4.2]{FM} for a proof. Before stating it, we need some definitions. First, we recall that a {\em multitwist} in $\Mod_g^n$ is an element that may be written as a product of powers of Dehn twists along a set of pairwise-disjoint simple closed curves on $S_g^n$. Next, an element of $\pi_1(S_g)$ is {\em simple} if it can be realized on $S_g$ as a loop without {\em transverse} self-intersections; in particular, note that any power of a simple loop is simple, according to the definition. Armed with these definitions, we have: 

\begin{proposition}
Let $c \in \pi_1(S_g)$ be a simple loop. Then $p(c)$ is a multitwist in $\Mod_g^1$. 
\label{thm:kra}
\end{proposition}

In fact, it is possible to give a more concrete description of $p(c)$ in the theorem above. Indeed, suppose that $c=a^n$, for some $a\in \pi_1(S_g)$ {\em primitive}. Let $a^{\pm}$ be the boundary components of a regular neighbourhood of (a representative of) $a$ in $S_g^1$. Then $p(a)= T_{a^+}^{} T_{a^-}^{-1}$, and therefore $p(c)= T_{a^+}^n T_{a^-}^{-n}$. 

\subsubsection{Algebraic version of the Birman exact sequence} Given a group $G$, denote by $\Aut(G)$ its automorphism group, and by $\Out(G)$ its {\em outer} automorphism group, namely the group of conjugacy classes of automorphisms of $G$. The following is the celebrated {\em Dehn-Nielsen-Baer Theorem} (see, for instance, \cite{FM}):

\begin{theorem}[Dehn-Nielsen-Baer]
For every $g\ge 1$, \[\Mod_g \cong \Out(\pi_1(S_g)).\]
\end{theorem}

In light of this theorem, the exact sequence \eqref{eq:birman} takes the following form: 

\begin{equation}
1\to \pi_1(S_g) \to \Aut(\pi_1(S_g)) \to \Out(\pi_1(S_g)) \to 1 
\label{eq:birmanalg}
\end{equation}

\section{Normal subgroups via covers}

In this section we will prove Theorem \ref{thm:main}. 
For an element $h$ of a group $G$ denote by $\langle \langle h\rangle\rangle_G$ the smallest normal subgroup of $G$  
containing $h$; we will drop $G$ from the notation when it is clear from the context.  
The main ingredient will be the following lemma:

\begin{lemma}  
Suppose there is an injective homomorphism \[\phi: \Mod_g^1 \to \Mod_{g'}^1\] and an element $h \in \Mod_g^1$ such that: 
\begin{enumerate}
\item $\proj^p_g(h) \in \Mod_g^1/\Mod_g^1[p]$ has infinite order.
\item $\phi(h)$ is a multitwist in $\Mod_{g'}^1$. In particular, $\phi(h^p) \in \ker(\proj^p_{g'})$. 
\item The  map \[\proj^p_{g'} \circ \hspace{.1cm} \phi: \Mod_g^1 \to \Mod_{g'}^1/\Mod_{g'}^1[p]\] has infinite image. 
\end{enumerate}
Then $\langle \langle \proj^p_g(h^p) \rangle \rangle$ is an infinite normal subgroup of $\Mod_g^1/\Mod_g^1[p]$ of infinite index. 
\label{lem:mainlemma}
\end{lemma}

\begin{proof}
Write $N = \langle \langle \proj^p_g(h^p) \rangle \rangle$, which is obviously a normal subgroup of $\Mod_g^1/\Mod_g^1[p]$. Observe that $N$ is infinite, by (1). Next, $\phi(h)$ is a multitwist, by (2), so in particular  $\phi(h^p)=\phi(h)^p \in \ker(\proj^p_{g'})$. Consider \[\proj^p_{g'} \circ \; \phi: \Mod_g^1 \to \Mod_{g'}^1/\Mod_{g'}^1[p],\] which again by (2)  factors through the quotient  \[\left( \Mod_g^1/\Mod_g^1[p]\right )/N.\] Finally, this quotient is infinite since $\proj_{g'}^p \circ \hspace{.1cm} \phi$ has infinite image, by (3). 
\end{proof}

We can now prove Theorem \ref{thm:main}:

\begin{proof}[Proof of Theorem \ref{thm:main}]
Let $g\ge 2$. We first explain how to construct the group $Q_1$ from the statement. 

\medskip

First, there exists $h_1\in \pi_1(S_g)$ of infinite order in $\Mod_g^1/\Mod_g^1[p]$: this is a consequence of work by Koberda-Santharoubane \cite[Theorem 4.1]{KS} for large $p$, and Funar-Lochak \cite[Proof of Proposition 3.2]{FL}  
for all $p$ as in the hypotheses of Theorem \ref{thm:main}.

We are going  to produce an injective homomorphism $\Mod_g^1 \to \Mod_{g'}^1$ as in Lemma \ref{lem:mainlemma}, suited to this element $h_1$. (We remark that our arguments are heavily inspired by the construction of the {\em non-geometric} injective homomorphism between mapping class groups of \cite{ALS}.)

First, a result of Scott \cite{Scott} implies that there exists a finite-degree cover $S_{g''} \to S_g$ to which $h$ {\em lifts simply}; that is, there exists some $n \in \N$ such that $h_1^n\in \pi_1(S_{g''})$ and $h_1^n$ has no transverse self-intersections.  Now, the intersection $K$ of all subgroups of $\pi_1(S_g)$ of index $[\pi_1(S_g): \pi_1(S_{g''})]$ is a characteristic subgroup of $\pi_1(S_g)$. Denote by $S_{g'} \to S_g$ the cover corresponding to $K=\pi_1(S_{g'})$, and observe that perhaps by raising $h_1^n$ to a higher power, we can assume that $h_1^n\in  \pi_1(S_{g'})$. We stress that, viewed as an element of $\pi_1(S_{g'})$, the element $h_1^n$ would still be simple. 

Next, since $\pi_1(S_{g'})$ is a characteristic subgroup of $\pi_1(S_g)$  we may define a homomorphism \[\phi: \Aut(\pi_1(S_g)) \to \Aut(\pi_1(S_{g'}))\] by assigning, to every $f \in \Aut(\pi_1(S_g))$, its restriction to $\pi_1(S_{g'})$. Note that $\phi$ is injective since $\pi_1(S_{g'})$ has finite index in $\pi_1(S_g)$ and surface groups have unique roots.

We now check that conditions (1)--(3) of Lemma \ref{lem:mainlemma} are satisfied for the homomorphism $\phi$ and the element $h=h_1^n$ constructed above. First, (1) holds since we chose $h_1$ of infinite order in  $\Mod_g^1/\Mod_g^1[p]$. Next, as $h_1^n$ is a simple element of $\pi_1(S_{g'})$,   Theorem \ref{thm:kra} implies that $\phi(h_1^n)$ is a multitwist, yielding (2). Finally, in order to prove (3), observe that $\phi(\pi_1(S_{g'})) = \pi_1(S_{g'})$, by construction. In particular, $\phi(\pi_1(S_{g}))$ contains $\pi_1(S_{g'})$ as a subgroup of finite index and we can again apply the result of Koberda-Santharoubane \cite[Theorem 4.1]{KS} used above to deduce that (3) holds. 
At this point, Lemma \ref{lem:mainlemma} tells us that the group \[Q_1:=\langle\langle \proj^p_g(h^p) \rangle \rangle\] is an infinite normal subgroup of infinite index of $\Mod_g^1/\Mod_g^1[p]$. 

\medskip

In order to construct the rest of the groups in the desired descending normal series, we apply the above argument inductively. More concretely, since 
\[\proj^p_{g'} \circ \hspace{.1cm} \phi:\Mod_g^1 \to \frac{\Mod_{g'}^1}{\Mod_{g'}^1[p]}\] 
has infinite image, there exists $h_2 \in \pi_1(S_g)$ whose image in 
\[(\Mod_g^1/\Mod_g^1[p])/Q_1\] has infinite order.
 Now, since $\phi(h_2) \in \phi(\pi_1(S_g))$ and the latter  contains $\pi_1(S_{g'})$ as a subgroup of finite index, up to taking a power of $h_2$ we may assume that $\phi(h_2) \in \pi_1(S_{g'})$. We now apply the above arguments to $\phi(h_2)$ to construct an injective homomorphism \[\phi_2: \Mod_{g'}^1\to \Mod_{g_2}^1\] such that $\phi_2(\phi(h_2))$ is a multitwist; in particular, $\phi_2(\phi(h_2^p)) \in \ker(\proj^{p}_{g_2})$. Again by the arguments above, the map \[ \proj^{p}_{g_2} \circ \; \phi_2 \circ \phi\]  has infinite image, and therefore $\langle \langle \proj^p_g(h_2^{p}) \rangle\rangle$ is an infinite normal subgroup of infinite index in 
 $(\Mod_g^1/\Mod_g^1[p])/Q_1$. We then set \[Q_2:= \langle \langle \proj^p(h^p), \proj^p(h_2^{p}) \rangle\rangle.\] Repeating this process we construct the desired descending normal series $Q_1 \trianglerighteq Q_2\trianglerighteq \ldots$.
\end{proof}

\section{A Birman exact sequence for quotient subgroups}  
Before proceeding with the proof of Proposition \ref{thm:surface} we need the following result 
providing sufficient conditions for the exactness of a Birman sequence associated to quotient groups, whose proof will be postponed a few lines: 
\begin{lemma}\label{lem:exactness}
Let $\Pi\subset \pi_1(S_g)$ be a characteristic subgroup, $K^1\subset \Mod_g^1$ be a normal subgroup 
and $K\subset \Mod_g$ denote its image by the forgetful map. Assume that: 
\begin{equation}
K^1\subset \ker\left(\Mod_g^1\to \Aut\left(\frac{\pi_1(S_g)}{\Pi}\right)\right)
\label{cong}
\end{equation}
\begin{equation}
\Pi \subset \pi_1(S_g) \cap  K^1
\label{pi}
\end{equation}
Then the following generalized Birman sequence is exact:
\[ 1 \to A\to \frac{\pi_1(S_g)}{\Pi} \to \frac{\Mod_g^1}{K^1} \to \frac{\Mod_g}{K}\to 1\]
for some central subgroup $A$, namely $A\subseteq Z\left(\frac{\pi_1(S_g)}{\Pi}\right)$.
\end{lemma}

Recall that $\pi_1(S_g)[p]$ denotes the subgroup of $\pi_1(S_g)$ generated by $p$-powers of {\em simple} elements of $S_g$; observe that $\pi_1(S_g)[p]$ is a characteristic subgroup of $\pi_1(S_g)$. 
We also need the following lemma, whose proof is postponed a few lines:

\begin{lemma}\label{lem:dehn}
The natural action of $\Mod_g^1[p]$ on $\frac{\pi_1(S_g)}{\pi_1(S_g)[p]}$ is trivial.
\end{lemma}

\begin{proof}[Proof of Proposition \ref{thm:surface}]
We want to apply the result of Lemma \ref{lem:exactness} to $\Pi=\pi_1(S_g)[p]$ and $K^1=\Mod_g^1[p]$. 
Using the description of the image of the point-pushing homomorphism for simple loops (see comment after Proposition \ref{thm:kra}), we know that $\pi_1(S_g)[p]\subset \pi_1(S_g)\cap \Mod_{g}^1[p]$, namely 
condition (\ref{pi}) holds. Further Lemma \ref{lem:dehn} shows that condition (\ref{cong}) is also satisfied and hence the claim follows. 
\end{proof}

\begin{proof}[Proof of Lemma \ref{lem:dehn}]
We start with some well-known observations, which are similar to \cite[Theorem 6.17]{C}. 
Let $a$ be a simple loop on $S_g$, viewed as an element of $\pi_1(S_g)$. If $a$ separates $S_g$ into two surfaces 
$S_{h,1}$ and $S_{g-h,1}$, then $\pi_1(S_g)=\pi_1(S_{h,1})*_{\Z} \pi_1(S_{g-h,1})$, where $\Z$ is the cyclic group generated by $a$. The orientation of the surface defines a right side, say $S_{h,1}$ and a left side for $a$. 
Then the (right) Dehn twist $T_{a}$ along $a$ is the automorphism determined by: 
\[ T_{a}(x)=\left\{\begin{array}{ll}
x, & {\rm if } \; x\in \pi_1(S_{h,1});\\
a x a^{-1}, & {\rm if } \; x\in \pi_1(S_{g-h,1});\\
\end{array}
\right.
\]
Thus 
\[ T_{a}^p(x)=\left\{\begin{array}{ll}
x, & {\rm if } \; x\in \pi_1(S_{h,1});\\
a^p x a^{-p}, & {\rm if } \; x\in \pi_1(S_{g-h,1});\\
\end{array}
\right.
\]
so that $T_{a}^p(x)\in x \pi_g[p]$, for any $x$. 
 
If $a$ is nonseparating, then we have a HNN splitting 
$\pi_1(S_g)=\pi_1(S_{g-1,2}) *_{\Z}$, where $\Z$ is generated by some element $t$ corresponding to a primitive (hence simple) loop 
intersecting $a$ once. Further the  Dehn twist $T_{a}$ along $a$ is the automorphism determined by: 
\[ T_{a}(x)=\left\{\begin{array}{ll}
x, & {\rm if } \; x\in \pi_1(S_{g-1,2});\\
t a, & {\rm if } \; x=t;\\
\end{array}
\right.
\]
Then 
\[ T_{a}^p(x)=\left\{\begin{array}{ll}
x, & {\rm if } \; x\in \pi_1(S_{g-1,2});\\
t a^p, & {\rm if } \; x=t;\\
\end{array}
\right.
\]
and so $T_{a}^p(x)\in x \pi_g[p]$, for any $x$. 
This finishes the proof of the lemma.
\end{proof}

We are now ready to prove Lemma \ref{lem:exactness}:

\begin{proof}[Proof of Lemma \ref{lem:exactness}]
By hypothesis, there exist well-defined homomorphisms:
\begin{equation}
\psi: \frac{\Mod_{g}^1}{K^1} \to \Aut\left(\frac{\pi_1(S_g)}{\Pi}\right)
\label{eq:p}
\end{equation}
and 
\begin{equation}
\overline{\psi}:\frac{\Mod_{g}}{K} \to \Out\left(\frac{\pi_1(S_g)}{\Pi}\right),
\label{eq:p2}
\end{equation}
which fit in a commutative diagram 

\begin{equation}
\xymatrix{
\frac{\Mod_{g}^1}{K^1} \ar[r]^{f} \ar[d]^{\psi} &
\frac{\Mod_{g}}{K} \ar[d]^{\overline{\psi}} \\
\Aut\left(\frac{\pi_1(S_g)}{\Pi}\right) \ar[r]^{q} & \Out\left(\frac{\pi_1(S_g)}{\Pi}\right).
}
\label{eq:diag1}
\end{equation}
In the above diagram, the homomorphism $f$ is the one induced by the forgetful homomorphism 
of Section \ref{sec:ses}, and the homomorphism $q$ is the obvious quotient map. We claim: 

\medskip

\noindent{\bf Claim.} The image of $ \frac{\pi_1(S_g)}{\pi_1(S_g) \cap K^1} < \frac{\Mod_{g}^1}{K^1} $ under the homomorphism $q \circ \psi$ is trivial. 

\begin{proof}[Proof of claim.]
First, the exactness of the Birman sequence \eqref{eq:birman} implies that \[f \left( \frac{\pi_1(S_g)}{\pi_1(S_g)\cap K^1} \right )\] is trivial, and therefore \[(\overline{\psi}\circ f) \left( \frac{\pi_1(S_g)}{\pi_1(S_g)\cap K^1} \right )\] is trivial too. Since the diagram \eqref{eq:diag1} is commutative, we deduce that 
\begin{equation}
\psi \left( \frac{\pi_1(S_g)}{\pi_1(S_g)\cap K^1} \right ) < \ker(q)= \frac{\frac{\pi_1(S_g)}{\Pi}}{Z\left(\frac{\pi_1(S_g)}{\Pi}\right)}, 
\end{equation}
which in particular implies the desired result. 
\end{proof}
From the claim above, we deduce that there is a diagram

\[\begin{array}{ccccccccc}
 1& \to & \pi_1(S_g) & \to & \Mod_g^1&\to & \Mod_g &\to& 1\\
  &  & \downarrow & & \downarrow & & \downarrow & & \\
 1& \to & \frac{\pi_1(S_g)}{\pi_1(S_g)\cap K^1} & \to & \frac{\Mod_{g}^1}{K^1}&\to & 
 \frac{\Mod_{g}}{K} &\to& 1\\
  &  & \downarrow & & \downarrow & & \downarrow & & \\
 1 & \to & \frac{\frac{\pi_1(S_g)}{\Pi}}{Z\left(\frac{\pi_1(S_g)}{\Pi}\right)} & \to & \Aut\left(\frac{\pi_1(S_g)}{\Pi}\right) & \to &\Out\left(\frac{\pi_1(S_g)}{\Pi}\right) & \to & 1 \\
 \end{array}
 \] 
 with exact rows.
  From condition (\ref{pi})  we have a surjective homomorphism 
   \begin{equation}
\alpha: \frac{\pi_1(S_g)}{\Pi}\to \frac{\pi_1(S_g)}{\pi_1(S_g)\cap K^1} 
\label{eq:q}
\end{equation} 
    It follows that that upper leftmost vertical arrow in the diagram above, which is the projection homomorphism 
\[\pi_1(S_g)\to \frac{\pi_1(S_g)}{\pi_1(S_g)\cap K^1},\] 
factors as
\[\pi_1(S_g)\to \frac{\pi_1(S_g)}{\Pi}\stackrel{\alpha}{\to} \frac{\pi_1(S_g)}{\pi_1(S_g)\cap K^1}.\] 
 Now, note that the composition of the two leftmost vertical arrows in the diagram above is the natural projection 
\[\pi_1(S_g)\to \frac{\frac{\pi_1(S_g)}{\Pi}}{Z\left(\frac{\pi_1(S_g)}{\Pi}\right)},\] 
which then factors as the composition of  homomorphisms: 
\[\pi_1(S_g)\to  \frac{\pi_1(S_g)}{\Pi}\stackrel{\alpha}{\to} \frac{\pi_1(S_g)}{\pi_1(S_g)\cap K^1}\stackrel{\psi}{\to} \frac{\frac{\pi_1(S_g)}{\Pi}}{Z\left(\frac{\pi_1(S_g)}{\Pi}\right)}\] 
Therefore  the projection map \[\frac{\pi_1(S_g)}{\Pi}\to \frac{\frac{\pi_1(S_g)}{\Pi}}{Z\left(\frac{\pi_1(S_g)}{\Pi}\right)}\] factors as 
\[\frac{\pi_1(S_g)}{\Pi}\stackrel{\alpha}{\to} \frac{\pi_1(S_g)}{\pi_1(S_g)\cap K^1}\stackrel{\psi}{\to} \frac{\frac{\pi_1(S_g)}{\Pi}}{Z\left(\frac{\pi_1(S_g)}{\Pi}\right)}.\]
Since $\alpha$ is surjective, there are induced homomorphisms $\alpha^\sharp$ and $\psi^\sharp$
such that the composition 
\[\frac{\frac{\pi_1(S_g)}{\Pi}}{Z\left(\frac{\pi_1(S_g)}{\Pi}\right)}\stackrel{\alpha^\sharp}{\to} 
\frac{\frac{\pi_1(S_g)}{\pi_1(S_g)\cap K^1}}{Z\left(\frac{\pi_1(S_g)}{\pi_1(S_g)\cap K^1}\right)}
\stackrel{\psi^\sharp}{\to} \frac{\frac{\pi_1(S_g)}{\Pi}}{Z\left(\frac{\pi_1(S_g)}{\Pi}\right)}\]  
is the identity.  Observe that  $\alpha^\sharp$ is also surjective, which implies that both maps  $\alpha^\sharp$ and $\psi^\sharp$ 
are isomorphisms. Therefore $\ker(\alpha)$ is a subgroup of $Z\left(\frac{\pi_1(S_g)}{\Pi}\right)$. 
Then the claim follows. 
\end{proof}

Recall that $\pi_1(S_g)[p]$ denotes the subgroup of $\pi_1(S_g)$ generated by $p$-powers of {\em simple} elements of $S_g$; observe that $\pi_1(S_g)[p]$ is a characteristic subgroup of $\pi_1(S_g)$. 
We also need the following lemma:

\begin{lemma}\label{lem:dehn}
The natural action of $\Mod_g^1[p]$ on $\frac{\pi_1(S_g)}{\pi_1(S_g)[p]}$ is trivial.
\end{lemma}

We will denote by $S_{g,n}$ the orientable surface of genus $g$ with $n$ boundary components.

\begin{proof}
We start with some well-known observations, which are similar to \cite[Theorem 6.17]{C}. 
Let $a$ be a simple loop on $S_g$, viewed as an element of $\pi_1(S_g)$. If $a$ separates $S_g$ into two surfaces 
$S_{h,1}$ and $S_{g-h,1}$, then $\pi_1(S_g)=\pi_1(S_{h,1})*_{\Z} \pi_1(S_{g-h,1})$, where $\Z$ is the cyclic group generated by $a$. The orientation of the surface defines a right side, say $S_{h,1}$ and a left side for $a$. 
Then the (right) Dehn twist $T_{a}$ along $a$ is the automorphism determined by: 
\[ T_{a}(x)=\left\{\begin{array}{ll}
x, & {\rm if } \; x\in \pi_1(S_{h,1});\\
a x a^{-1}, & {\rm if } \; x\in \pi_1(S_{g-h,1});\\
\end{array}
\right.
\]
Thus 
\[ T_{a}^p(x)=\left\{\begin{array}{ll}
x, & {\rm if } \; x\in \pi_1(S_{h,1});\\
a^p x a^{-p}, & {\rm if } \; x\in \pi_1(S_{g-h,1});\\
\end{array}
\right.
\]
so that $T_{a}^p(x)\in x \pi_g[p]$, for any $x$. 
 
If $a$ is nonseparating, then we have a HNN splitting 
$\pi_1(S_g)=\pi_1(S_{g-1,2}) *_{\Z}$, where $\Z$ is generated by some element $t$ corresponding to a primitive (hence simple) loop 
intersecting $a$ once. Further the  Dehn twist $T_{a}$ along $a$ is the automorphism determined by: 
\[ T_{a}(x)=\left\{\begin{array}{ll}
x, & {\rm if } \; x\in \pi_1(S_{g-1,2});\\
t a, & {\rm if } \; x=t;\\
\end{array}
\right.
\]
Then 
\[ T_{a}^p(x)=\left\{\begin{array}{ll}
x, & {\rm if } \; x\in \pi_1(S_{g-1,2});\\
t a^p, & {\rm if } \; x=t;\\
\end{array}
\right.
\]
and so $T_{a}^p(x)\in x \pi_g[p]$, for any $x$. 
This finishes the proof of the lemma.
\end{proof}

We are now in a position to prove Proposition  \ref{thm:surface} and Corollary \ref{cor:surfintermediary}.

\begin{proof}[Proof of Proposition \ref{thm:surface}]
We want to apply the result of Lemma \ref{lem:exactness} to $\Pi=\pi_1(S_g)[p]$ and $K^1=\Mod_g^1[p]$. 
Using the description of the image of the point-pushing homomorphism for simple loops (see comment after Proposition \ref{thm:kra}), we know that $\pi_1(S_g)[p]\subset \pi_1(S_g)\cap \Mod_{g}^1[p]$, namely 
condition (\ref{pi}) holds. Further Lemma \ref{lem:dehn} shows that condition (\ref{cong}) is also satisfied and hence the claim follows. 
\end{proof}

\begin{proof}[Proof of Corollary \ref{cor:nonlatticepi}]
In the proof of Theorem \ref{thm:main}, all of the elements $h_i$ constructed belong to $\pi_1(S_g)$. 
Then the subgroups 
\[ \overline{Q}_i = Q_i \cap \frac{\pi_1(S_g)}{\pi_1(S_g)[p]}\]
are normal subgroups of $\frac{\pi_1(S_g)}{\pi_1(S_g)[p]}$ thus forming a descending normal series. 
Each $\overline{Q}_i$ is infinite and of infinite index into $\overline{Q}_{i-1}$ by the proof above, based on the 
Koberda-Santharoubane infiniteness statement  \cite[Theorem 1.1]{KS} .  We have: 
\[\langle\langle \proj^p_g(h_1^p),\ldots,\proj^p_g(h_j^{p})\rangle\rangle
\subset \frac{\pi_1(S_g)}{\pi_1(S_g)[p]}.\]
In fact, the groups \[\langle\langle \proj^p_g(h_1^p),\ldots,\proj^p_g(h_j^{p})\rangle\rangle
\] are characteristic subgroups of $\frac{\pi_1(S_g)}{\pi_1(S_g)[p]}$, and in particular they are also normal subgroups 
in $\frac{\Mod{g}^1}{\Mod_{g}^1[p]}$. 
\end{proof}
\begin{remark}
The subgroups 
 \[\langle\langle \proj^p_g(h_1^p),\ldots,\proj^p_g(h_j^{p})\rangle\rangle_{\Mod_g^1}\cap \pi_1(S_g)\] 
are normal subgroups in $\Mod_{g}^1$, in particular they form a series of characteristic 
subgroups of $\pi_1(S_g)$. 
\end{remark}

\section{Virtually K\"ahler groups}
In this section we prove Theorem \ref{thm:kahler}. We refer the reader to \cite{FM,Hubbard}, as well as to the references below, for standard facts about moduli space.

 The analytic moduli space $\mathcal M_g^{an}$  of smooth genus $g$ algebraic curves is 
an orbifold with orbifold fundamental group $\Mod_g$. It is a classical result of Serre \cite{Serre} that 
$\mathcal M_g^{an}$ has finite-degree unramified coverings which are topological manifolds, e.g.  those
obtained by quotienting  Teichm\"uller space by the stabilizer of an abelian level $k\geq 3$ structure, namely 
\[ \ker(\Mod_g\to \Aut(H_1(S_g,\Z/k\Z)).\] 
There is a standard K\"ahler structure on Teichm\"uller space, namely the one induced by the Weil-Petersson metric. In particular, each of the finite-degree unramified covers of $\mathcal M_g^{an}$ mentioned above also admits a K\"ahler structure. However, it should be noted that this fact does not imply that $\Mod_g$ is virtually K\"ahler, as the action of the relevant group on the appropriate cover of  $\mathcal M_g^{an}$ is not cocompact.

The Deligne-Mumford compactification $\overline{\mathcal M}_g^{an}$ of 
$\mathcal M_g$ is the analytic space underlying the moduli stack $\overline{\mathcal M}_g$ of 
stable curves of genus $g$. Looijenga \cite{Loo} and later several other authors 
proved that $\overline{\mathcal M}_g ^{an}$ also admits finite Galois covers 
which are topological manifolds.

Deligne and Mumford \cite{DM} introduced  the moduli stack  $_G{\mathcal M}_g$ of smooth genus $g$ algebraic 
curves with a Teichm\"uller structure of level $G$, where $G$ is any finite characteristic quotient 
$G$ of $\pi_1(S_g)$, and considered its underlying moduli space $_G\mathcal M_g^{an}$, which in turn parametrizes 
smooth curves with a $G$-structure. They further defined the compactification 
$_G\overline{\mathcal M}_g^{an}$  of $_GM_g^{an}$ as the normalization of  $\overline{\mathcal M}_g$ 
with respect to  $_G\mathcal M_g^{an}$.

For suitable choices of the finite group $G$, the 
analytic space  $_G\overline{\mathcal M}_g^{an}$ is a  compact  smooth complex manifold (see 
\cite{Boggi,BP,Loo,PJ}). Moreover, the  map  forgetting the level structure $_G\overline{\mathcal M}_g^{an}\to \overline{\mathcal M}_g^{an}$  is a Galois covering ramified along the divisor at infinity. 
In order to prove Theorem \ref{thm:kahler} we will compute 
$\pi_1(_G\overline{\mathcal M}_g^{an})$ in such cases. 
For the sake of conciseness, most arguments borrowed from  \cite{Boggi,BP,Loo,PJ} are only sketched below. 

We consider, after Pikaart-Jong \cite{PJ},  the family of characteristic subgroups 
\[ \pi_g(k,p)=\gamma_{k+1}(\pi_1(S_g))\cdot  \pi_1(S_g)^p,\]
where $\pi_1(S_g)^p$ denotes the subgroup of $\pi_1(S_g)$ generated by all $p$-th powers, and $\gamma_k$ denotes the $k$-th term of the lower central series of a group $G$, namely the one defined as 
$\gamma_1(G)=G$ and $\gamma_{k+1}(G)=[\gamma_k(G),G]$, for $k\geq 1$. 
 We stress that the groups $\pi_g(k,p)$ are finite index subgroups of $\pi_1(S_g)$, as nilpotent groups with generators of finite order are finite. We set:
\[\Mod_g(k,p)= \ker(\Mod_g\to \Out(\pi_1(S_g)/\pi_g(p,k))\]

Now, the following result is  due to Pikaart-Jong \cite{PJ}:

\begin{theorem}[\cite{PJ}]
The analytic space $_{\pi_g(p,k)}\overline{\mathcal M}_g^{an}$ is smooth 
if $k\geq 4$, $p\geq 3$ and $\gcd(p,6)=1$, if $k=1$ and $g=2$, if $k=2$ and $p$ is odd, 
if $k=3$ and $p$ is either odd or divisible by $4$.   
\end{theorem}

Note that $\pi_1(_{\pi_g(k,p)}{\mathcal M}_g^{an})=\Mod_g(k,p)$. In order to compute the fundamental 
group $\pi_1(_{\pi_g(k,p)}\overline{\mathcal M}_g^{an})$ of the compactification  
it suffices to know the monodromy along the boundary 
on the relative fundamental group of the universal curve over $\mathcal M_g^{an}$. 
In fact the orbifold homotopy class of a  loop in $\mathcal M_g^{an}$ coincides with the 
class of the monodromy of the universal curve along that loop, when the orbifold fundamental group of $\mathcal M_g^{an}$ is identified with $\Mod_g$.

Let $C$ be a complex stable genus $g$ curve with singular points $p_1,\ldots,p_m$. 
The stable curve $C$ is the singular fiber of a local universal deformation family of stable curves over the polydisk in $\C^{3g-3}$.  Choose local coordinates  $z_i$ such that $z_j=0$ describes stable curves for which  $p_j$ is  a singular point.  Then, the universal family has smooth fibers outside the discriminant locus given by the equation: 
\[ z_1 z_2 \cdots z_m =0\]

Vanishing cycles in the universal family correspond to a set $c_1,c_2,\ldots,c_m$ of 
simple closed curves on the generic smooth fiber, which are shrunk to the singular points 
$p_1,p_2,\ldots,p_m$ of the singular fiber. If $U$ denotes the complement of the discriminant, then
$\pi_1(U)$ is the abelian group generated by classes $[\gamma_j]$ of loops  $\gamma_j$ encircling the complex hyperplane 
$z_j=0$ exactly once, and in the positive direction. Moreover, the monodromy $\rho$  of the universal curve is given by  
\[ \rho([\gamma_i])=T_{c_i}\]
where $T_{c_i}$ denotes the right Dehn twist along the curve $c_i$. 
In other words, the orbifold homotopy classes of the loops $[\gamma_i]$ in $\mathcal M_g^{an}$ 
are the classes $T_{c_i}$.

Let $D\subset \overline{\mathcal M}_g^{an}$ be the divisor at infinity.  Recall  (\cite{DM}, Thm.5.2)  that $D$ is a normal crossing divisor in the orbifold sense, see (\cite{ACGH} chap. XII) for a detailed description both from orbifold and stack point of view. 
Note that $U$ is homeomorphic to the intersection of 
$ \overline{\mathcal M}_g^{an}-D$ with a neighbourhood of the  point $[C]$ within $\overline{\mathcal M}_g^{an}$. 
A similar description exists for neighbourhoods of points of $_{\pi_g(k,p)}\overline{\mathcal M}_g^{an}$ lying 
above $[C]$.
  
We will use then the following lemma (see \cite[Section 7]{Fox}):
\begin{lemma}\label{lem:fox}
Let $Y$ be a connected barycentrically-subdivided locally-finite complex, and let $K$ be a subcomplex such that, for each vertex $v$ of $K$, the intersection $S(v)$ of $Y-K$ with the open star of $v$ in $Y$ is nonempty and connected. 
Then the  homomorphism  $\pi_1(Y-K)\to \pi_1(Y)$ 
induced by the natural inclusion is surjective. Moreover, its kernel is the normal closure of those elements in $\pi_1(Y-K)$ that 
can be represented by loops in $\cup_v S(v)$.  
\end{lemma}

With the above lemma to hand, we now choose a triangulation of $\overline{\mathcal M}_g^{an}$ for which $D$ is a subcomplex, 
which in turn induces a triangulation of $_{\pi_g(k,p)}\overline{\mathcal M}_g^{an}$. 
We want to use Lemma \ref{lem:fox}, by taking for $Y$ the  triangulation $_{\pi_g(k,p)}\overline{\mathcal M}_g^{an}$, and for $K$ 
the triangulation of the divisor at infinity.  

Note that the $S(v)$ are pairwise disjoint in the statement of Lemma \ref{lem:fox}. As such, each $S(v)$ is homeomorphic to $U$. 
It remains to identify the classes of the lifted  loops $\gamma_j$ in $\pi_1(_{\pi_g(k,p)}\overline{\mathcal M}_g^{an})$. 
Equivalently, we need to find  the kernel of the local monodromy  representation for $_{\pi_g(k,p)}\overline{\mathcal M}_g^{an}$ 
over a neighbourhood of $[C]$ in $\overline{\mathcal M}_g^{an}$. 
This has been already done in \cite[Thm. 3.1.3]{PJ}:

\begin{theorem}[\cite{PJ}]\label{thm:monodromy}
If $k\geq 4$ and $\gcd(p,6)=1$, then 
the kernel of the monodromy is generated by the classes of the form $T_{c_i}^p$. 
\end{theorem}

Before proving our result, we will need the following lemma:

\begin{lemma}\label{lem:normal}
Let $K<\Mod_g$ be a normal subgroup, with the property that  $T_c^n\in K$, for every simple closed curve $c$. 
Set $\langle\langle T_{c}^n\rangle\rangle_K$ for the normal subgroup of $K$ 
generated by the $n$-th powers of all Dehn twists. Then  $\langle\langle T_{c}^n\rangle\rangle_K= K\cap \Mod_g[n]$. 
\end{lemma}

\begin{proof}[Proof of Lemma \ref{lem:normal}]
We have $\langle\langle T_{c}^n\rangle\rangle_K\subset  K\cap \Mod_g[n]$.
For the reverse inclusion consider $x\in K\cap \Mod_g[n]$. Thus 
$x=\prod g_i T_{c_i}^{\varepsilon_i n}g_i^{-1}$, where $g_i\in \Mod_g$, 
$\varepsilon_i\in\{-1,1\}$. 
But $ g_i T_{c_i}^{\varepsilon_i n}g_i^{-1}= 
T_{g_i(c_i)}^{\varepsilon_i n}$. Recall that $T_{c}^n\in K$, for every 
simple closed curve $c$. Therefore 
$x$ is the product of the $n$-th powers of Dehn twists 
$T_{g_i(c_i)}^{\varepsilon_i n}$, each element being in $K$. 
This means that $x\in \langle\langle T_{c}^n\rangle\rangle_K$. 
\end{proof}

We are finally in a position for proving the main result of this section: 

\begin{proof}[Proof of Theorem \ref{thm:kahler}]
By Lemma \ref{lem:fox}  the group $\pi_1(_{\pi_g(k,p)}\overline{\mathcal M}_g^{an})$ is isomorphic to the quotient 
\[\Mod_g(k,p)/\langle\langle T_{\gamma}^p\rangle\rangle_{\Mod_g(k,p)}\]
Using Lemma \ref{lem:normal} we have 
\[ \langle\langle T_{\gamma}^p\rangle\rangle_{\Mod_g(k,p)}=
\Mod_g(k,p)\cap \Mod_g[p]\]
so that 
\[ \pi_1(_{\pi_g(k,p)}\overline{\mathcal M}_g^{an})=\frac{\Mod_g(k,p)}{\Mod_g(k,p)\cap \Mod_g[p]}\]
is a subgroup of finite index in  $\frac{\Mod_g}{\Mod_g[p]}$. This proves the claim.
\end{proof}

\begin{proof}[Proof of Corollary \ref{cor:rank1}]
Since $\frac{\Mod_g}{\Mod_g[p]}$ is virtually K\"ahler, the results are consequences of corresponding statements of 
Carlson-Toledo \cite{CT} for lattices in $\SO(n,1)$, and of Simpson \cite{Simpson} for groups of Hodge type. 
\end{proof}

Before  proceeding with the proof we will need the following result due to Carlson-Toledo (see  \cite{CT}, Thm. 7.1 and Cor. 3.7):
\begin{theorem}[\cite{CT}]\label{thm:Carlson-Toledo}
Let  $F:\mathcal X\to \mathbb H^{n+1}_{\R}/\Lambda$ be a harmonic map from a compact K\"ahler manifold $\mathcal X$ to 
a hyperbolic space form $\mathbb H^{n+1}_{\R}/\Lambda$, where $\Lambda\subset SO(1,n)$, for some $n >2$, is a lattice.
Then the harmonic map $F$ factors either through a circle or else through a compact Riemann surface. 
\end{theorem}

\begin{proof}[Proof of Corollary \ref{cor:homom}]
Let $J$ be a finite index K\"ahler subgroup of $\frac{\Mod_g}{\Mod_g[p]}$; for instance, after the proof of Theorem \ref{thm:kahler}, we may set 
$$J=\frac{\Mod_g(k,p)}{\Mod_g(k,p)\cap \Mod_g[p]}$$ Let $\mathcal X$ be a 
compact K\"ahler manifold with fundamental group $J$, and \[f:\frac{\Mod_g}{\Mod_g[p]}\to \Lambda\] a homomorphism 
into a lattice $\Lambda\subset SO(1,n)$, for some $n >2$.  Note that $f$ is 
induced by a map $F:\mathcal X\to \mathbb H^{n+1}_{\R}/\Lambda$ into a hyperbolic space form. 
Eells-Sampson \cite{ES} proved that then the map $F$ could be assumed to be 
a harmonic map.  Further,  by Theorem  \ref{thm:Carlson-Toledo} the 
harmonic map $F$ factors  through a circle or a compact Riemann surface. 
Thus $f$ factors through $\Z$ or through $\pi_1(\Sigma_h)$. 
In the first case $f(J)$ is either trivial, or isomorphic to $\Z$.  In the second case, the subgroup $f(J)\subset \pi_1(\Sigma_h)$ is finitely generated and hence also
a surface group, by a result of Scott \cite{Scott}. 

Now, if $J$ were rationally perfect, then  $f(J)$ 
would be finite, by a result of Scott \cite{Scott}, since surface groups with finite abelianization 
are trivial. A recent result of Ershov-He (\cite{EH}, Thm. 1.9) shows that $H_1(J;\Q)=0$ for any finite index 
subgroup $J\subset \Mod_g$ with the property that $J\subset \gamma_k \mathcal{I}_g$ and $4g \geq 8 k-4$, where 
$\mathcal{I}_g$ denotes the Torelli group. In particular, this is so if $J$ contains the $k$-th Johnson subgroup
\[ \ker\left(\Mod_g \to \Out\left(\frac{\pi_1(S_g)}{\gamma_{k+1}(\pi_1(S_g))}\right)\right)\] 
Note that $\Mod_g(k,p)$ contains the $k$-th Johnson subgroup, by definition. 
\end{proof}

\begin{remark}
All examples worked out in \cite[Section 3]{BP} provide finite ramified coverings with finite or trivial 
$\pi_1(_G\overline{\mathcal M}_g^{an})$. 
\end{remark}

\begin{remark}
The same method yields (many) other virtually K\"ahler quotients of $\Mod_g$. For instance we could take 
for odd $p$ the normal subgroup generated by $p$-th powers of separating bounded simple closed curves of genus $1$ and 
the $3p$-th powers of other Dehn twists, see (\cite{BP}, Prop. 2.8, $k=7$) and also 
(\cite{Boggi}, Prop. 3.12) for other possibilities.
\end{remark}

\section{Abelianization of finite index subgroups}

Finally, we prove Theorem \ref{prop:abelianization}: 

\begin{proof}[Proof of Theorem \ref{prop:abelianization}]
As before, let $T_{c}$ denote the  Dehn twist along the simple closed curve $c$. 
By hypothesis $T_{c}^n\in K$. Let $p:K\to H_1(K, \mathbb{Q})$ be 
the abelianization map.  
According to Bridson \cite{Br} and Putman \cite{Pu}
we have  $p(T_{c}^n)=0$, which implies that $p$ factors through $K/\langle\langle T_{c}^n\rangle\rangle_K$.  
Observe now
that the inclusion $K\to \Mod_g$ induces an injective homomorphism 
\[ K(n)=\frac{K}{\langle\langle T_{c}^n\rangle\rangle_K}=\frac{K}{K\cap\Mod_g[n]}\hookrightarrow \frac{\Mod_g}{\Mod_g[n]}\]
Moreover, we also have that 
\[ \left[\frac{\Mod_g}{\Mod_g[n]}:K(n)\right]=
[\Mod_g: K]=n,\] and we are done.
 \end{proof}

Finally, we remark that the normality assumption for $K$ in Proposition \ref{prop:abelianization} is not essential. 
In fact, we have: 

\begin{proposition}
Suppose there exists $K\subset \Mod_g$ of finite index 
with  $\dim H_1(K;\Q) >0$, then there exists a normal subgroup 
$H\subset\Mod_g $, also of finite index, with $\dim H_1(H;\Q) >0$. 
\end{proposition}
\begin{proof}
It is known that if $H\subset K$ is a subgroup of finite index then 
the map $H_1(H,\Q)\to H_1(K,\Q)$ is surjective. This follows from the 
existence of the tranfer map in homology (see \cite{Pu}, Lemma 2.1). 
It suffices now to consider $H$ a normal subgroup of $\Mod_g$ contained in $K$. 
\end{proof}

\bigskip


\bigskip

\end{document}